\documentclass [a4paper,reqno, 11pt]{amsart}
\usepackage[english]{babel}

\usepackage{fullpage}
\usepackage{amsmath, amsfonts, amsthm}
\usepackage{graphicx}
\usepackage[colorlinks=true, allcolors=blue]{hyperref}
\usepackage{mathtools} 
\allowdisplaybreaks 
\usepackage{comment} 
\usepackage[capitalize]{cleveref}
\usepackage{bbm}
\usepackage{nicematrix}
\usepackage{tikz}
\usepackage{array}
\usepackage{wrapfig}
\usepackage{combelow}

\usepackage{enumitem}
\usepackage{algorithm}
\usepackage{algorithmic}

\theoremstyle{plain}
\newtheorem{theorem}{Theorem}[section]

\theoremstyle{plain}

\theoremstyle{plain}
\newtheorem{proposition}[theorem]{Proposition}

\theoremstyle{plain}
\newtheorem{lemma}[theorem]{Lemma}

\theoremstyle{definition}
\newtheorem{definition}[theorem]{Definition}

\theoremstyle{definition}
\newtheorem{example}[theorem]{Example}

\theoremstyle{definition}
\newtheorem{remark}[theorem]{Remark}

\theoremstyle{definition}

\theoremstyle{definition}

\theoremstyle{definition}
\newtheorem{question}[theorem]{Question}

\usepackage[indent,skip=.2\baselineskip]{parskip}

\newcommand{\calM}{\mathcal{M}}
\newcommand{\calV}{\mathcal{V}}
\newcommand{\frakm}{\mathfrak{m}}

\DeclareMathOperator{\LT}{LT}

\newcommand{\arxiv}[1]{\href{http://arxiv.org/abs/#1}{{\tt arXiv:#1}}}

\title{Partial fraction decompositions\\ on hyperplane arrangements}
\author{Claire de Korte, Teresa Yu}

\date{\today}

\begin{document}

\begin{abstract}
    We study partial fraction decompositions (PFDs) in several variables using tools from commutative algebra. We give criteria for when a rational function with poles on a hyperplane arrangement has a desirable PFD. Our criteria are obtained by examining the primary decomposition of ideals coming from hyperplane arrangements.
    We then present an algorithm for finding a PFD that satisfies properties desired for simplifying the calculation of scattering amplitudes. We demonstrate the effectiveness of this algorithm by computing practical examples coming from Feynman integrals.
\end{abstract}

\maketitle

\section{Introduction}\label{sec:intro}

Partial fraction decompositions (PFDs) of rational functions in several variables appear frequently in high-energy physics. Examples include abstract applications, such as the recasting of scattering amplitudes as canonical forms of polytopes \cite{ABL}, and computational applications like developing efficient algorithms to compute Feynman integrals \cite{HvM}. Despite their ubiquity in physics, studies of PFDs seldom appear in the mathematical literature \cite{BBD+,Bec,Lei}. There has been interest from physicists to utilize algebro-geometric techniques in order to compute certain PFDs \cite{BWW+,CDE,DP}.
In this paper we study PFDs from the perspective of commutative algebra, by using primary decomposition and algorithms based upon Gr\"obner bases. This is done by recasting the existence of PFDs as an ideal membership problem.

Let $R=K[x_1,\ldots,x_r]$ be a polynomial ring over a field $K$, and let $f,g\in R$ be nonzero polynomials. We are interested in \emph{partial fraction decompositions (PFDs)} of the rational function $f/g$; these are expressions of the function as:
\begin{equation*}
        \frac{f}{g}=\sum_{i=1}^m\frac{f_i}{g_i},\qquad \deg(f_i)<\deg(f),\,\deg(g_i)<\deg(g).
    \end{equation*}

\begin{example}\label{eg:intro}
To see where the complexity in PFDs can arise, consider the following example of a rational function with two different PFDs:
\begin{equation*}
\begin{split}
    \frac{13y-6x}{(y-3x)(x+y)(x-2y)}&=\frac{19}{12x(x+y)}-\frac{2}{15x(x-2y)}-\frac{33}{20x(y-3x)}\\
    &=\frac{1}{(x+y)(x-2y)}+\frac{2}{(y-3x)(x-2y)}-\frac{5}{(y-3x)(x+y)}.
\end{split}
\end{equation*}
\end{example}
\noindent This example demonstrates that:
\begin{enumerate}
    \item rational functions do not necessarily have a unique PFD; and
    \item certain PFDs of rational functions can introduce \emph{spurious poles}, i.e., poles in the terms of the PFD that do not appear in the rational function itself (see Definition~\ref{defn:spur pole} for a precise definition).
\end{enumerate}

In applications, one would like to know when a rational function $f/g$ has a PFD, and if so, how to make the PFD ``optimal." Here, optimal can refer to a number of parameters, including minimizing the number of terms $m$ in the PFD, maximizing the degrees of the numerators $f_i$, and avoiding spurious poles (the introduction of such poles can complicate calculations and obscure physical properties and symmetries \cite{Hod}).

In this paper, we ask the following question: given a rational function $f/g$ such that $g=\ell_1\cdots\ell_n$ can be factored into linear polynomials in $R$, can we determine when an optimal PFD for $f/g$ exists? We desire a combinatorial answer in terms of the hyperplane arrangement associated with the linear polynomials $\ell_1,\ldots,\ell_n$. We say that such a rational function $f/g$ has a \emph{PFD (of degree $d$)} if it can be expressed as:
\begin{equation}\label{eqn:d PFD}
    \frac{f}{g}=\frac{f}{\ell_1\cdots\ell_n}=\sum_{\{j_1,\ldots,j_{n-d}\}\subset[n]} \frac{h_{j_1\cdots j_{n-d}}}{\ell_{j_1}\cdots\ell_{j_{n-d}}},\qquad \deg(\ell_i)=1,\,\deg(h_{j_1\cdots j_{n-d}})\leq \deg(f)-d.
\end{equation}

The main results (Theorems~\ref{thm: PFD iff proj} and \ref{thm: PFD iff affine}) of this paper characterize the answer to this question in terms of the numerator $f$'s geometric properties relative to the hyperplane arrangement; the main result for hyperplane arrangements in projective space is the following.

\begin{theorem}\label{thm:intro-homog}
    Let $f\in K[x_1,\ldots,x_r]$ be a homogeneous polynomial of degree $\deg(f)\geq d\geq 1$, and let $\ell_1,\ldots,\ell_n$ be nonzero linear forms with $n\geq d$. 
    Then $f/(\ell_1\cdots\ell_n)$ has a PFD of degree $d$ as in (\ref{eqn:d PFD}) if and only if the numerator $f$ vanishes to order at least $d-n+|S|$ on all flats $S$ of size at least $n-d+1$ of the hyperplane arrangement associated with the linear forms. 
\end{theorem}

Our approach to proving Theorem~\ref{thm:intro-homog} is with commutative algebra: given linear forms $\ell_1,\ldots,\ell_n$ of the polynomial ring $R=K[x_1,\ldots,x_r]$, we use the primary decomposition of the following ideal:
\begin{equation}\label{eq:THE ideal}
    I_{L,d}=\langle \ell_{i_1}\cdots\ell_{i_d}:1\leq i_1<\cdots<i_d\leq n\rangle.
\end{equation}
Existence of a PFD of order $d$ for the rational function $f/(\ell_1\cdots\ell_n)$ is equivalent to the numerator $f$ being in the ideal $I_{L,d}$, and a primary decomposition of $I_{L,d}$ gives an equivalent geometric condition for when this occurs.

One can use the computer algebra software \texttt{Macaulay2} \cite{M2} to determine the existence of and algorithmically compute PFDs of the form in 
(\ref{eqn:d PFD}). Our algorithm (Algorithm~\ref{algo:PFD}) for finding a PFD is based upon Gr\"obner basis algorithms. It satisfies conditions \ref{wishlist(i)}, \ref{wishlist(ii)}, \ref{wishlist(iii)}, and \ref{wishlist(iv)} of the ``wishlist'' for a good partial fraction algorithm, described by Heller and von Manteuffel in \cite[\S 2.3]{HvM} and based upon work of Le\u{\i}nartas \cite{Lei}:
\begin{enumerate}[label=(\roman*), ref=(\roman*)]
    \item \label{wishlist(i)}  the algorithm provides a unique answer;
    \item  \label{wishlist(ii)} the algorithm does not introduce spurious denominator factors;
    \item \label{wishlist(iii)} the algorithm commutes with summation;
    \item \label{wishlist(iv)} the algorithm eliminates any spurious poles that are present in the input.
\end{enumerate}

We now discuss the organization and contributions of this paper. In \S\ref{sec: primary decompositions}, we discuss the ideals $I_{L,d}$ in (\ref{eq:THE ideal}) defined from hyperplane arrangements. We present primary decompositions for these ideals (Theorems~\ref{thm:flats-pri-decomp} and \ref{thm:affine-pri-decomp}), which are given in terms of the matroid associated with the hyperplane arrangement. 
We then apply these decompositions in \S\ref{sec:braid} to study the ideals $I_{L,d}$ where $L$ is the braid arrangement. In \S\ref{sec: polytope adjoint polynomial}, we also use the primary decomposition to deduce vanishing behavior of adjoint polynomials arising from polytopes; these polynomials are objects of much study in physics, combinatorics, and algebraic geometry. We then return to our main motivation of PFDs in \S\ref{sec: PFDs}: in this section, we give criteria for when PFDs of degree $d$ exist (Theorems~\ref{thm: PFD iff proj} and \ref{thm: PFD iff affine}). We present Algorithms~\ref{algo:reduced} and \ref{algo:PFD} to determine the existence of and compute PFDs; we also demonstrate how our algorithms satisfy all four conditions \ref{wishlist(i)}--\ref{wishlist(iv)} of a good PFD algorithm. In \S\ref{sec:physics}, we then show how our criteria and algorithm can be applied to PFDs coming from two different contexts from physics: Feynman integrals (\S\ref{subsec:feynman}) and flat-space wavefunction coefficients (\S\ref{subsec:wavefunc}). We conclude in \S\ref{sec:future} with remarks on the outlook and directions for future investigation for the algebraic study of PFDs.

Code for the algorithms and examples presented in this paper are publicly available at:

\begin{center}
    \small
    \url{https://github.com/CJMDK/PFDs}. 
\end{center}

\begin{remark}
    After initially posting our paper, we were informed that Theorem~\ref{thm:flats-pri-decomp} was already proven in \cite{BTX, Toh}. We thank an anonymous referee and \cb{S}tefan Toh\u{a}neanu for referring us to these results. Our primary contribution, connecting PFDs to commutative algebra, remains original. 
\end{remark}

\section{Primary decompositions}\label{sec: primary decompositions}

We fix a field $K$, let $r,n\geq 1$ be arbitrary positive integers, and let $R=K[x_1,\ldots,x_r]$ denote the polynomial ring in $r$ variables. The central geometric objects of this paper are hyperplane arrangements, defined by $n$ linear polynomials $L=(\ell_1,\ldots,\ell_n)$, $\ell_i\in R$; we consider arrangements in both affine and projective spaces, with the $\ell_i$ homogeneous in the latter setting.

Let $d$ be a positive integer with $1\leq d\leq n$. Given a hyperplane arrangement $L=(\ell_1,\ldots,\ell_n)$, we consider the ideal $I_{L,d}\subset R$ defined by $d$-fold products of the linear forms:
\begin{equation*}
I_{L,d}=\langle\ell_{i_1}\cdots\ell_{i_d}:1\leq i_1<\cdots<i_d\leq n\rangle.
\end{equation*}
We allow for the cases that some of the $\ell_i=0$ are zero, or that the $\ell_i$ may not all be distinct, i.e., $\ell_i=\ell_j$ for for some $i\neq j$. However, we assume that $I_{L,d}\neq 0$ is not the zero ideal, so not all of the $\ell_i=0$ are zero. The purpose of this section is to present primary decompositions for $I_{L,d}$, first in the case of a generic hyperplane arrangement and then for arbitrary arrangements.

\subsection{Generic hyperplane arrangements in projective space}\label{subsec:generic primary decomp}

Before tackling the case of an arbitrary hyperplane arrangement, we focus first on \emph{generic} arrangements in projective space, i.e., the hyperplanes are in linearly general position; the matroid corresponding to the arrangement is the \emph{uniform matroid} (see Remark~\ref{rmk:uniform matroid}).
Here, we let $n\geq r$ be arbitrary positive integers, and fix $d=n-r+1$.

The main result of this subsection (Theorem~\ref{thm:generic-m^d}) shows that $I_{L,d}=\mathfrak{m}^d$ is the $d$th power of the homogeneous maximal ideal $\mathfrak{m}\subset R$. This was previously proven by Toh\u{a}neanu in the context of coding theory \cite[Theorem 3.1]{Toh}. We give an alternative proof.

To state the following proposition, we introduce some notation. Write each linear form $\ell_i\in R$ in the form $\ell_i=\sum_{j=1}^r a_{ij}x_j$, with $a_{ij}\in K$. Define $A=(a_{ij})$ to be the $(n\times r)$-matrix of these coefficients. For $1\leq i_1<\cdots<i_r\leq n$, denote by $A[i_1,\ldots, i_r]$ the $r$-minor of $A$ obtained by taking the submatrix of $A$ with rows labeled by $\ell_{i_1},\ldots,\ell_{i_r}$. Now, define $\calM$ to be the $(\binom{n}{d}\times\binom{n}{d})$-matrix of coefficients, with rows labeled by $d$-fold products $\ell_{i_1}\cdots\ell_{i_d}$ of the linear forms, and columns labeled by degree $d$ monomials of $R$.

\begin{proposition}\label{prop: linear independence of products}
With the notation set above, we have:
\begin{equation}
\det(\mathcal{M})=\pm\prod_{1\leq i_1<\cdots<i_r\leq n}A[i_1,\ldots, i_r]. 
\end{equation}
In particular, if every $r$-subset of the linear forms $\ell_1,\ldots,\ell_n$ is linearly independent, then the set of $d$-fold products of the linear forms is also linearly independent.
\end{proposition}

\begin{proof}
 We may assume that $K$ is algebraically closed. Note that each \(A[i_1,\ldots, i_r]\) is irreducible, and so by Hilbert's Nullstellensatz,
\[A[i_1,\ldots, i_r]\text{ divides } \det(\mathcal{M}) \ \Leftrightarrow \ \mathcal{V}(A[i_1,\ldots, i_r])\subset \mathcal{V}(\det(\mathcal{M})).\]

We now show that $\mathcal{V}(A[i_1,\ldots, i_r])\subset \mathcal{V}(\det(\mathcal{M}))$. Suppose that there exists some relation between the \(\ell_{i}\), i.e., for some \(\{i_1,\ldots,i_r\}\subset[n]\), there exists \(c_1,\ldots, c_r\in K\) not all zero such that:
\begin{equation}\label{eq: linear dependence relation}
\sum_{j=1}^{r}c_i\ell_{i_j}=0.
\end{equation}
The existence of such a relation is equivalent to \(A[i_1,\ldots, i_r]\) being identically zero. In other words, the tuples \((a_{11},\ldots,a_{nr})\in K^{n\times r}\) for which such a relation (\ref{eq: linear dependence relation}) exist form \(\mathcal{V}(A[i_1,\ldots, i_r])\).
Define the set \(\{i'_1,\ldots,i'_{d-1}\}=[n]\setminus\{i_1,\ldots,i_r\}\). Multiplying each term in (\ref{eq: linear dependence relation}) by \(\ell_{i'_1}\cdots\ell_{i'_{d-1}}\), we find that:
\begin{equation}\label{eq: product linear independence relation}
\sum_{j=1}^{r}c_i\ell_{i_j}\ell_{i'_1}\cdots\ell_{i'_{d-1}}=0.
\end{equation}
Each \(c_i\) in (\ref{eq: product linear independence relation}) is multiplied by a unique collection of \(d\) linear forms.
But then we see that the set of $d$-fold products of the $\ell_i$ is linearly dependent, and so $\det(\calM)$ vanishes on \(\mathcal{V}(A[i_1,\ldots, i_r])\). Thus \(\mathcal{V}(A[i_1,\ldots, i_r])\subset \mathcal{V}(\det(\mathcal{M}))\). This implies that:
\[\det(\mathcal{M})=h\cdot\prod_{\substack{\{i_1,\ldots, i_r\}\subset[n] \\ i_i<\cdots<i_r}}A[i_1,\ldots, i_r],\qquad h\in K[a_{11},\ldots,a_{nr}].\]

Each $A[i_1,\ldots, i_r]$ is a degree $r$ polynomial, and so their product has degree $\binom{n}{r}\cdot r$ in the $a_{ij}$.
The entries of the matrix \(\mathcal{M}\) are each of degree \(d\) in the \(a_{ij}\), and so \(\det(\mathcal{M})\) has degree \(\binom{n}{d}\cdot d\). For \(n=r+d-1\), we have \(\binom{n}{r}\cdot r=\binom{n}{d}\cdot d\), so \(\det(\mathcal{M})\) and the product of $A[i_1,\ldots, i_r]$ are polynomials of the same degree. This is only possible if \(h\in K\). The fact that $h=\pm 1$ is proven in the
following lemma (Lemma \ref{lem:LT-generic}).
\end{proof}

\begin{lemma}\label{lem:LT-generic}
    Following the notation of the proof of Proposition \ref{prop: linear independence of products}, we have that $h=\pm 1$:
    \begin{equation}\label{eqn:lem-LT}
        \det(\mathcal{M})=\pm\prod_{1\leq i_i<\cdots<i_r\leq n}A[i_1,\ldots, i_r].
    \end{equation}\
\end{lemma}

\begin{proof}
    Consider a diagonal term order on the polynomial ring $K[a_{11},\ldots,a_{nr}]$. Then for any $r$-subset $\{i_1,\ldots,i_r\}\subset[n]$, the leading term (with coefficient) of $A[i_1,\ldots, i_r]$ is
    \[\LT(A[i_1,\ldots, i_r])=\pm a_{i_1 1}\cdots a_{i_r r}.\]
    Therefore, the leading term of the RHS of (\ref{eqn:lem-LT}) is $\pm m$, where $m$ is the monomial with the exponent of the variable $a_{ij}$ given by: 
    \begin{itemize}
        \item if $i\geq j$: $\binom{i-1}{j-1}\binom{n-i}{r-j}$, which is the number of $r$-sequences $(i_1,\ldots,i_r)$ where $1\leq i_1<\cdots<i_r\leq n$ and $i$ appears in the $j^{\text{th}}$ spot;
        \item if $i<j$: $0$.
    \end{itemize}

    Now consider the matrix $\mathcal{M}$. Its columns are indexed by degree $d$ monomials in $R$, so order its columns by lexicographic order. Its rows are indexed by $r$-subsets of $[n]$, so order these by lexicographic order on $[n]^r$. Then the leading term $m'$ of $\det(\mathcal{M})$ is the product of terms appearing in the diagonal of $\mathcal{M}$. Using the ``stars and bars" combinatorial argument, i.e., the correspondence between monomials and $n$-tuples of stars and bars, with $i_1,\ldots,i_r$ corresponding to the indices where stars are placed, one sees that $m'$ is the monomial with the exponent of the variable $a_{ij}$ given by:
    \begin{itemize}
        \item if $i\geq j$: $\binom{i-1}{j-1}\binom{n-i}{r-j}$, which is the number of $n$-tuples of stars and bars where the $i^{\text{th}}$ position is a star appearing after $j-1$ bars and before $r-j$ bars;
        \item if $i<j$: $0$.
    \end{itemize}
    Thus, we see that the leading terms of both sides of (\ref{eqn:lem-LT}) are the same, up to multiplication by $\pm 1$. This implies the desired result.
\end{proof}

\begin{theorem}\label{thm:generic-m^d}
    Let $L=(\ell_1,\ldots,\ell_n)$ be $n$ linear forms in $R$ such that every $r$-subset of the \(\ell_i\) is linearly independent, and let $d=n-r+1$. 
    Then $I_{L,d}=\mathfrak{m}^d$, where $\mathfrak{m}=\langle x_1,\ldots,x_r\rangle$ is the homogeneous maximal ideal. In particular, the set $\{\ell_{i_1}\cdots\ell_{i_d}:1\leq i_1<\cdots<i_d\leq n\}$ forms a $K$-basis of the vector space of degree $d$ homogeneous polynomials in $R$.
\end{theorem}

\begin{proof}
    By Proposition \ref{prop: linear independence of products}, the set \(\{\ell_{i_1}\ell_{i_2}\cdots\ell_{i_d}:1\leq i_1<\cdots<i_d\leq n\}\) is linearly independent, and has cardinality $\binom{n}{d}$. The vector space of homogeneous degree $d$ polynomials in $R$ also has dimension $\binom{n}{d}$, and so the set of $d$-fold products of the $\ell_i$ forms a basis for this vector space.
    Consequently, the ideal $I_{L,d}$ is the same as the ideal generated by all degree $d$ polynomials of the polynomial ring; this latter ideal is exactly $\mathfrak{m}^d$.
\end{proof}

\begin{remark}\label{rmk:uniform matroid}
The description of the linear forms \(\ell_i\) in Theorem~\ref{thm:generic-m^d} can be reformulated from a matroid-theoretic perspective, with the configuration of the hyperplane arrangement corresponding to the uniform matroid of rank \(r\) on \(n\) elements. Explicitly, setting \(\{\ell_1,\ldots,\ell_n\}\) as the ground set, the collection of independent sets is:
\[\{\{\ell_{i_1},\ldots,\ell_{i_m}\}:\{i_1,\ldots,i_m\}\subset[n], \ m\le r\}.\]
The flats of this matroid are the subsets of the ground set of cardinality at most $r$, together with the ground set itself. This foreshadows results of the following subsection, where we discuss arbitrary hyperplane arrangements.
\end{remark}

\subsection{Arbitrary hyperplane arrangements}\label{subsec:affine}

In this subsection, we give a combinatorial description for a primary decomposition of $I_{L,d}$. This is given in terms of a matroid associated with the hyperplane arrangement.

We first focus on the case where $L$ is a hyperplane arrangement in projective space.
Suppose $L=(\ell_1,\ldots,\ell_n)$ are linear forms in $R$.
The associated matroid $M$ is on the ground set $[n]=\{1,\ldots,n\}$. The independence sets of this matroid are the subsets $\{i_1,\ldots,i_m\}\subset[n]$ such that the linear forms $\ell_{i_1},\ldots,\ell_{i_m}$ are linearly independent in the vector space of linear forms in $R$.

Given such a matroid $M$, we have an associated \emph{rank function} $\mathrm{rk}(-)$ on $M$: for $S\subset[n]$, we define $\mathrm{rk}(S)$ to be the rank of the $K$-vector space spanned by $\{\ell_i:i\in S\}$. For a subset $S\subset[n]$, we define the \emph{closure} $\mathrm{cl}(S)$ of $S$ to be:
\[\mathrm{cl}(S)=\{i\in[n]:\mathrm{rk}(S)=\mathrm{rk}(S\cup\{i\})\}.\]
A \emph{flat} of the matroid $M$ is defined to be a subset $S\subset[n]$ such that $S=\mathrm{cl}(S)$.

The primary decompositions we consider are defined in terms of powers of the prime ideals
\[I_S=\langle\ell_i:i\in S\rangle,\qquad S\subset[n].\]
These primary decompositions are given in \cite[Theorem 2.2]{BTX}. The following is a reformulation of their result.

\begin{theorem}\label{thm:flats-pri-decomp}
    Let $M$ denote the matroid corresponding to the linear forms $L=(\ell_1,\ldots,\ell_n)$. Then we have the following primary decomposition of the ideal $I_{L,d}$:
    \[I_{L,d}=\bigcap_{\substack{S\subset[n],\\ |S|\geq n-d+1, \\ \text{$S$ a flat of $M$}}}(I_S)^{d-n+|S|}.\]
\end{theorem}

We note that the primary decomposition in Theorem~\ref{thm:flats-pri-decomp} is not necessarily minimal, as the following example illustrates. One may try to use more refined data from the associated matroid to determine the associated primes and minimal primary decomposition of $I_{L,d}$, such as in \cite[Theorem 3.2]{CT}.

\begin{example}
    Let $r=3,d=4,n=6$. Consider the matrix of coefficients for the linear forms given by:
        \[\begin{pmatrix}0 & 1 & 1 & 1 & 0 & 0 \\
        0 & 0 & 0 & 0 & 1 & 0 \\
        0 & 0 & 0 & 0 & 0 & 1\end{pmatrix}.\]
        Then a minimal primary decomposition of the corresponding ideal $I_{L,d}$ is:
        \[I_{L,d}= \langle x\rangle^2\cap\langle y,z\rangle\cap\langle x,y\rangle^3\cap\langle x,z\rangle^3.\]
    Since $S=[n]$ is always a flat, the ideal $\mathfrak{m}^d=\langle x,y,z\rangle^4$ appears in the primary decomposition given in Theorem~\ref{thm:flats-pri-decomp}. However, in this example it is redundant and does not appear in the minimal primary decomposition.
    
    Consider now the following slightly modified matrix of coefficients:
        \[\begin{pmatrix}
            1 & 1 & 1 & 0 & 0 & 0\\
            0 & 0 & 0 & 1 & 1 & 0\\
            0 & 0 & 0 & 0 & 0 & 1
        \end{pmatrix}.\]
        Then the minimal primary decomposition of the corresponding ideal $I_{L,d}$ is:
        \[I_{L,d} = \langle x\rangle\cap\langle x,y\rangle^3\cap\langle x,z\rangle^2\cap\langle y,z\rangle\cap\langle x,y,z\rangle^4,\]
        which agrees with the primary decomposition given in Theorem~\ref{thm:flats-pri-decomp}.
\end{example}

From the homogeneous case, we may now deduce results for primary decompositions of hyperplane arrangements in affine space.

Consider linear polynomials $\ell_1,\ldots,\ell_n\in R$. By homogenizing, we get linear forms $\ell_1',\ldots,\ell_n'\in R'=K[x_1,\ldots,x_{r+1}]$. Let $I_{L,d}\subset R$ denote the ideal of $d$-fold products of the $\ell_i$, and let $J_{L,d}\subset R'$ denote the ideal of $d$-fold products of the $\ell_i'$. Then: 
\[I_{L,d}=(J_{L,d}+\langle x_{r+1}-1\rangle)\cap R,\]
where $J_{L,d}$, via Theorem~\ref{thm:flats-pri-decomp}, has a primary decomposition in terms of powers of ideals $J_S=\langle \ell_i':i\in S\rangle\subset R'$.
This enables us to obtain the following primary decomposition of $I_{L,d}$ in $R$:

\begin{theorem}\label{thm:affine-pri-decomp}
    Let $L=(\ell_1,\ldots,\ell_n)$ with $\ell_i\in R$ be linear polynomials. Then we have the following primary decomposition of $I_{L,d}$:
    \[I_{L,d}=\bigcap (I_S)^{d-n+|S|},\]
    where the intersection is over subsets $S\subset[n]$ such that: 
    \begin{enumerate}
        \item $S$ indexes a flat of the hyperplane arrangement associated with the homogenizations of the $\ell_i$;
        \item the hyperplane at infinity is not in the span of this associated flat, i.e., $x_{r+1}\notin J_S$; and
        \item $|S|\geq n-d+1$.
    \end{enumerate}
\end{theorem}

\begin{proof}
    We retain the notation introduced in the paragraph preceding the statement of the theorem. By Theorem~\ref{thm:flats-pri-decomp}, we have that:
    \begin{align*}
        I_{L,d} &= (J_{L,d}+\langle x_{r+1}-1\rangle)\cap R\\
        &=\left(\left(\bigcap_{\substack{\text{$S$ a flat for $\ell'_i$},\\|S|\geq n-d+1}} (J_S)^{d-n+|S|}\right)+\langle x_{r+1}-1\rangle\right)\cap R\\
        &=\bigcap_{\substack{\text{$S$ a flat for $\ell_i'$},\\|S|\geq n-d+1}}\left((J_S)^{d-n+|S|}+\langle x_{r+1}-1\rangle\right)\cap R.
    \end{align*}
    If $J_S$ is supported on the hyperplane at infinity, i.e., $x_{r+1}\in J_S$, then: 
    \[\left((J_S)^{d-n+|S|}+\langle x_{r+1}-1\rangle\right)\cap R=R,\]
    and so the term is redundant in the intersection. Otherwise:
    \[\left((J_S)^{d-n+|S|}+\langle x_{r+1}-1\rangle\right)\cap R=(I_S)^{d-n+|S|},\]
    which gives the desired equality.
\end{proof}

\begin{example}
    Consider the lines defined by:
    \[\ell_1=x,\quad\ell_2= y, \quad \ell_3=x-1,\quad \ell_4=y-1,\quad \ell_5=x-y\]
    in the polynomial ring $K[x,y]$. Their homogenizations in $K[x,y,z]$ are:
    \[\ell_1'=\ell_1,\quad\ell_2'=\ell_2,\quad\ell_3'=x-z,\quad\ell_4'=y-z,\quad\ell_5'=x-y.\]
    A primary decomposition for the ideal $I_{L,d}$ of $(d=3)$-fold products of the $\ell_i$ is given by:
    \[I_{L,d}=\langle x,y\rangle\cap\langle x-1,y-1\rangle.\]
    This can be obtained from a primary decomposition for the ideal $J_{L,d}$ of $3$-fold products of the $\ell_i'$, removing ideals supported on the hyperplane $z=0$, and then setting $z=1$:
    \[J_{L,d}=\langle x,y\rangle\cap\langle x-z,y-z\rangle\cap\langle x,y,z\rangle^3.\]
\end{example}

\section{The braid arrangement}\label{sec:braid}

We apply the primary decomposition in Theorem~\ref{thm:flats-pri-decomp} to \emph{braid arrangements}, which form an important family of hyperplane arrangements arising across algebra, geometry, topology, and combinatorics (see, e.g., \cite{Fei}). Braid arrangements also appear in physics via denominators of \emph{broken Parke--Taylor functions} \cite{Fro} and \emph{Parke--Taylor forms} \cite{BEPV}.

\begin{definition}
The \textit{braid arrangement} $\mathcal{B}_r$ consists of the following set of $\binom{r}{2}$ hyperplanes in $R=K[x_1,\ldots,x_r]$:
\[\mathcal{B}_r=\{\calV(x_i-x_j): \ i,j\in[r], \ i<j\}.\] 
\end{definition}

The flats of the braid arrangement $\mathcal{B}_r$ are in bijection with partitions of the set $[r]$; if $B_1\sqcup\cdots\sqcup B_m=[r]$, then a flat of the arrangement is defined by the equations:
\[\{x_i-x_j:\text{$i<j$ for which $i,j\in B_k$ are in the same subset of the partition}\}.\]
For each flat of $\mathcal{B}_r$, one may therefore associate a partition $\lambda\vdash r$, i.e., a tuple $\lambda=(\lambda_1,\ldots,\lambda_r)$ where $\lambda_1+\cdots+\lambda_r=r$ and $\lambda_1\geq\cdots\lambda_r\geq 0$. Then, for a flat associated with the partition $\lambda$, the number of hyperplanes contained in the flat is given by:
\[\binom{\lambda_1}{2}+\binom{\lambda_2}{2}+\cdots+\binom{\lambda_r}{2}.\]
Table~\ref{table: partitions of braid arrangement} summarizes the flats and their associated partitions for $\mathcal{B}_5$.

\renewcommand{\arraystretch}{1.4}

\begin{table}[h!]

\begin{center}
    \small\begin{tabular}{l|l|l|l}
        Partition & Number of flats & Number of hyperplanes & Example flat  \\
        \hline
  $(5)$ & $\binom{5}{5}=1$ & $\binom{5}{2}=10$ & $12,13,14,15,23,24,25,34,35,45$ \\[5pt]
  $(4,1)$ & $\binom{5}{4}=5$ & $\binom{4}{2}=6$ & $12,13,14,23,24,34$\\[5pt]
  $(3,2)$ & $\binom{5}{3}=10$ & $\binom{3}{2}+\binom{2}{2}=4$ & $12,13,23,45$\\[5pt]
  $(3,1,1)$ & $\binom{5}{3}=10$ & $\binom{3}{2}=3$ & $12,13,23$\\[5pt]
  $(2,2,1)$ & $\binom{5}{2}+\binom{3}{2}=13$ & $\binom{2}{2}+\binom{2}{2}=2$ & $12,34$\\[5pt]
  $(2,1,1,1)$ & $\binom{5}{2}=10$ & $\binom{2}{2}=1$ & $12$\\[4pt]
  $(1,1,1,1,1)$ & 1 & 0 & --
    \end{tabular}
\end{center}
\vspace{5mm}
\caption{The first column lists all the partitions of $5$, the second states the number of flats associated with a given partition, the third gives the number of hyperplanes for each such flat, and the fourth gives an example of such a flat. We use $ij$ to denote the polynomial $x_i-x_j$ for brevity.}
\label{table: partitions of braid arrangement}
\end{table}

\begin{example}\label{eg:braid5}
    Let $r=5$ and $d=8$, and consider the braid arrangement $\mathcal{B}_5$. There are $n=\binom{5}{2}=10$ linear forms in our arrangement $L$. Then Theorem~\ref{thm:flats-pri-decomp} gives a primary decomposition of $I_{L,d}$ as an intersection of the following kinds of ideals:
    \begin{itemize}
        \item the homogeneous maximal ideal $\frakm$ to the power of $d=8$, corresponding to the single flat associated with $(5)$;
        \item five ideals, each to the power of $4$, corresponding to flats associated with $(4,1)$;
        \item ten ideals, each to the power of $2$, corresponding to flats associated with $(3,2)$;
        \item ten ideals, each to the power of $1$, corresponding to flats associated with $(3,1,1)$.
    \end{itemize}
\end{example}

One may notice from Table~\ref{table: partitions of braid arrangement} and Example~\ref{eg:braid5} that if $\lambda\succeq\mu$ under the dominance order on partitions (recall that $\lambda\succeq\mu$ under the \emph{dominance order} if and only if $\lambda_1+\cdots+\lambda_i\geq\mu_1+\cdots+\mu_i$ for all $i\in[r]$), then the size of a flat associated with $\lambda$ is larger than that of $\mu$. The following result shows that this holds more generally, thus showing that the set of partitions associated with flats appearing in the primary decomposition of Theorem~\ref{thm:flats-pri-decomp} is a filter in the poset of partitions.

\begin{proposition}\label{prop:braid}
    Suppose $\lambda,\mu\vdash r$ be partitions, with $\lambda\succeq\mu$ under the dominance order. Then $|S_\lambda|\geq |S_\mu|$, where $S_\lambda$ (resp. $S_\mu$) denotes any flat associated with $\lambda$ (resp. $\mu$). In particular, if flats corresponding to the partition $\mu$ appear in the primary decomposition of Theorem~\ref{thm:flats-pri-decomp}, then so do flats corresponding to the partition $\lambda$. 
\end{proposition}

\begin{proof}
    Let $\lambda=(\lambda_1,\ldots,\lambda_r)$ and $\mu=(\mu_1,\ldots,\mu_r)$, with $\lambda_i,\mu_i\geq 0$ for all $i$. It suffices to assume that $\lambda\succeq\mu$ is a cover relation. By \cite[Proposition 2.3]{Bry}, this implies that there exists $i,j\in [r]$ with $i<j$ such that $\lambda_i=\mu_i+1$, $\lambda_j=\mu_j-1$, $\lambda_i-\lambda_j\geq 2$, and $\lambda_k=\mu_k$ for all $k\neq i,j$. Then $|S_\lambda|\geq |S_\mu|$ if and only if:
    \[\binom{\lambda_i}{2}+\binom{\lambda_j}{2}\geq\binom{\mu_i}{2}+\binom{\mu_j}{2}=\binom{\lambda_i-1}{2}+\binom{\lambda_j+1}{2}.\]
Consequently:
    \begin{align*}
        \left(\binom{\lambda_i}{2}+\binom{\lambda_j}{2}\right)-\left(\binom{\lambda_i-1}{2}+\binom{\lambda_j+1}{2}\right)&=\lambda_i-\lambda_j-1\geq 1,
    \end{align*}
    and so the desired conclusion holds.
\end{proof}

\section{Adjoint polynomials}\label{sec: polytope adjoint polynomial}

A rapidly developing field at the intersection of physics, combinatorics, and algebraic geometry is the study of positive geometries and canonical forms \cite{Lam,RST}. Central mathematical objects in this study are projective polytopes and their \emph{adjoint polynomials}. In this section, we highlight how the primary decomposition in Theorems~\ref{thm:flats-pri-decomp} and \ref{thm:affine-pri-decomp} can be used to easily deduce vanishing properties of adjoint polynomials. We define this polynomial, and refer the reader to the paper \cite{Gae} for further introduction to the relevant terminology.

\begin{definition}\label{def: adjoint definition}
Let $P$ be a polytope in $\mathbb{R}^{r-1}$. Let $C\subset\mathbb{R}^r$ be the cone over $P$ with rays $V(C)$ corresponding to vertices of $P$. 
 For each simplicial cone $S\in T$ of a triangulation $T$ of $C$, denote by $a_S$ the volume of the parallelepiped whose edges are defined by the rays $V(S)$ of $S$. The \emph{adjoint polynomial} of $C$ is defined as:
\begin{equation*}
\mathrm{adj}_C(x)=\sum_{S\in T}\left(a_S\prod_{v\in V(C)\backslash V(S)}\ell_v(x)\right).
\end{equation*}
Here, $\ell_v(x)=v\cdot x$, where $\cdot$ denotes scalar product.
\end{definition}

The form of $\mathrm{adj}_C$ is independent of the choice of triangulation of $C$, and $\mathrm{adj}_C$ is homogeneous and of degree $|V(C)|-\dim(C)$ \cite{Gae, War}.
Furthermore, Kohn--Ranestad showed that the adjoint polynomial must vanish on the residual arrangement of the 
hyperplane arrangement defining the facets of the dual cone $C^\vee$; the adjoint polynomial is also uniquely characterized by certain vanishing conditions on this residual arrangement \cite{BW,KR}.

The above definition of adjoint polynomial makes it clear that adjoint polynomials are in ideals of the form $I_{L,d}$, where in this case the hyperplanes defining the ideal are defined by the $\ell_{v}$ from Definition~\ref{def: adjoint definition}. 
Consequently, using the primary decomposition of Theorem~\ref{thm:flats-pri-decomp}, one can determine certain properties about where adjoint polynomials vanish, without needing to compute any of the $a_S$ coefficients or the residual arrangement of the dual cone. 
Theorem~\ref{thm:flats-pri-decomp} implies the following vanishing behavior of adjoint polynomials.

\begin{proposition}
    Let $P\subset\mathbb{R}^{r-1}$ be a polytope, and let $C$ be the cone over $P$. Then the adjoint polynomial $\mathrm{adj}_C$ vanishes to order at least $d-n+|S|$ on all flats $S$ of size at least $n-d+1$ of the hyperplane arrangement given by the homogenized facet hyperplanes of $P^\vee$.
\end{proposition}

This vanishing condition is different from that of \cite[Theorem 1.5]{KR}, which guarantees vanishing of $\mathrm{adj}_C$ on the residual arrangement of the dual polytope.

\begin{figure}[h]
\centering
\begin{tikzpicture}[scale=0.2ex]
\draw[black] (0,0) -- (0,1);
\draw[black] (0,0) -- (1,-1);
\draw[black] (0,0) -- (1,1);
\draw[black] (0,1) -- (1,2);
\draw[black] (0,1) -- (1,-1);
\draw[black] (1,2) -- (1,1);
\draw[black] (1,2) -- (4,2);
\draw[black] (1,1) -- (4,1);
\draw[black] (4,2) -- (5,1);
\draw[black] (4,2) -- (4,1);
\draw[black] (4,1) -- (5,0);
\draw[black] (5,1) -- (5,0);
\draw[black] (5,0) -- (4,-1);
\draw[black] (5,1) -- (4,-1);
\draw[black] (1,-1) -- (4,-1);
\filldraw[black] (0,0) circle (2pt) node[scale=1.2, xshift=0, yshift=0]{};
\filldraw[black] (0,1) circle (2pt) node[scale=1.2, xshift=0, yshift=0]{};
\filldraw[black] (1,-1) circle (2pt) node[scale=1.2, xshift=0, yshift=0]{};
\filldraw[black] (1,2) circle (2pt) node[scale=1.2, xshift=0, yshift=0]{};
\filldraw[black] (1,1) circle (2pt) node[scale=1.2, xshift=0, yshift=0]{};
\filldraw[black] (4,1) circle (2pt) node[scale=1.2, xshift=0, yshift=0]{};
\filldraw[black] (4,2) circle (2pt) node[scale=1.2, xshift=0, yshift=0]{};
\filldraw[black] (5,1) circle (2pt) node[scale=1.2, xshift=0, yshift=0]{};
\filldraw[black] (5,0) circle (2pt) node[scale=1.2, xshift=0, yshift=0]{};
\filldraw[black] (4,-1) circle (2pt) node[scale=1.2, xshift=0, yshift=0]{};
\end{tikzpicture}
\caption{Polytope used to define the adjoint polynomial in Example~\ref{ex: polytope adjoint polynomial}.}
\label{fig: polytope}
\end{figure}
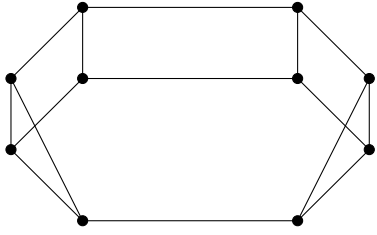

\begin{example}\label{ex: polytope adjoint polynomial}
Consider the polytope $P$ shown in Figure~\ref{fig: polytope}. Its homogeneous facet hyperplanes can be given by:
\begin{equation*}
\begin{aligned}
   \ell_1&=x_2-x_4,\\
   \ell_3&=-2x_1+x_2+2x_4,\\
   \ell_5&=-2x_1+x_2-2x_4,\\
   \ell_7&=x_2-2x_3+x_4.
\end{aligned}
\quad\quad
\begin{aligned}
    \ell_2&=2x_1+x_2-2x_4,\\
    \ell_4&=2x_1+x_2+2x_4,\\
    \ell_6&=x_3
\end{aligned}
\end{equation*}
Each of these can be written in the form $v_i\cdot x$ where the $v_i$ are the set of vertex rays of a pointed polyhedral cone $C$ (the dual cone associated with $P$). The adjoint polynomial of this cone is thus of the form:
\begin{equation*}
    \text{adj}_C(x)=\sum_{1\leq i_1<i_2<i_3\leq 7}a_{i_1 i_2i_3}\ell_{i_1}\ell_{i_2}\ell_{i_3},\qquad a_{i_1i_2i_3}\in\mathbb{R}.
\end{equation*}
This polynomial is an element of the ideal:
\[I_{L,3}=\langle\ell_{i_1}\ell_{i_2}\ell_{i_3}:1\le i_1<i_2<i_3\le7\rangle,\]
which, using Theorem~\ref{thm:flats-pri-decomp}, can be expressed as the primary decomposition:
\begin{equation*}
    I_{L,3}=\bigcap_{\substack{S\subset[7],\,|S|\ge5,\\\text{$S$ a flat of M}}}(I_S)^{|S|-4}=\langle\ell_i:i\in[5]\rangle\cap\langle\ell_i:i\in[7]\rangle^3.
\end{equation*}
Here, we have used that the matroid defined by the facet hyperplanes only has two flats containing more than four elements, those being $\{\ell_i:i\in[5]\}$ and $\{\ell_i:i\in[7]\}$. This result tells us that the adjoint polynomial associated with the polytope $P$ vanishes with order at least one where the five linear forms $\ell_1,\ldots,\ell_5$ all vanish.
\end{example}

\begin{remark}
It was shown in \cite{Gae} that canonical forms of polytopes $P$ can be expressed as rational functions, the numerator being the adjoint polynomial associated with the polytope $P^\vee$ dual to $P$, and the denominator being a product of facet hyperplanes of $P^\vee$. This means that canonical forms are exactly a type of rational function with a PFD of degree $d$.
\end{remark}

\section{Partial fraction decompositions and algorithmic aspects}\label{sec: PFDs}

We now return to the problem of existence and computation of PFDs. Let $f/(\ell_1\cdots\ell_n)$ be a rational function, where $f,\ell_1,\ldots,\ell_n\in R=K[x_1,\ldots,x_r]$ are nonzero polynomials such that $\deg(\ell_i)=1$. Recall that we are interested in when this function has a PFD of degree $d$, meaning that it can be expressed as:
\begin{equation}\label{eqn:pfd secalgo}
    \frac{f}{\ell_1\cdots\ell_n}=\sum_{\{j_1,\ldots,j_{n-d}\}\subset[n]} \frac{h_{j_1\cdots j_{n-d}}}{\ell_{j_1}\cdots\ell_{j_{n-d}}},\qquad \deg(h_{j_1\cdots j_{n-d}})\leq \deg(f)-d.
\end{equation}
Typically, the least common multiple of all of the terms in the PFD expression is exactly $\ell_1\cdots\ell_n$. By putting all of the terms of the PFD under this common denominator, we see that such a PFD exists exactly when the polynomial $f$ can be expressed as an $R$-linear combination of $d$-fold products of the $\ell_i$'s. In other words, a PFD of degree $d$ exists for $f/(\ell_1\cdots\ell_n)$ if and only if $f$ lies in the ideal $I_{L,d}$ presented in \S\ref{sec: primary decompositions}. We can thus interpret the primary decompositions of $I_{L,d}$ as vanishing criteria for existence of PFDs.
We first consider the case of hyperplane arrangements in linearly general position.

\begin{theorem}\label{thm: generic PFD}
Let $\ell_1 ,\ldots, \ell_n$ be $n$ linear forms in \(R=K[x_1,\ldots,x_r]\) such that any $r$-subset of the \(\ell_i\) are linearly independent. Set \(d=n-r+1\), and let \(f\in R\) be any homogeneous degree \(d\) polynomial. Then the rational function $f/(\ell_1\cdots\ell_n)$
has a unique PFD of degree $d$.
\end{theorem}

\begin{proof}
    By Theorem~\ref{thm:generic-m^d}, the polynomial \(f\) can be written uniquely as a $K$-linear combination of the elements of \(\{\ell_{i_1}\ell_{i_2}\cdots\ell_{i_d}:1\leq i_1<\cdots<i_d\leq n\}\):
\[f=\sum_{\substack{\{i_1,\ldots, i_d\}\subset[n], \\ i_1<\cdots<i_d}}a_{i_1\ldots i_d}\ell_{i_1}\cdots\ell_{i_d},\qquad a_{i_1\cdots i_d}\in K.\]
Thus, we can write the rational function as:
\[\frac{f}{\ell_1\ell_2\cdots\ell_n}=\sum_{\substack{\{i_1,\ldots, i_{d}\}\subset[n], \\ i_1<\cdots<i_{d}}}\frac{a_{i_1\cdots i_d}}{(\ell_1\ell_2\cdots\ell_n)/(\ell_{i_1}\cdots\ell_{i_d})}=\sum_{\substack{\{j_1,\ldots, j_{n-d}\}\subset[n], \\ j_1<\cdots<j_{n-d}}}\frac{b_{j_1\cdots j_{n-d}}}{\ell_{j_1}\cdots\ell_{j_{n-d}}},\]
where we set \(b_{j_1\cdots j_{n-d}}\) to be equal to \(a_{i_1\cdots i_d}\) with \(\{i_1,\ldots, i_d\}=[n]\setminus\{j_1,\ldots,j_{n-d}\}\).
\end{proof}

We now turn to the case of arbitrary hyperplane arrangements, both in the projective and affine settings. The following results on PFDs follow from Theorems~\ref{thm:flats-pri-decomp} and \ref{thm:affine-pri-decomp}.

\begin{theorem}\label{thm: PFD iff proj}
Let $\ell_1,\ldots,\ell_n$ be $n$ nonzero linear forms in $R$, let $d\leq n$, and let $f\in R$ be a homogeneous polynomial with $\deg(f)\geq d$. Then $f/(\ell_1\cdots\ell_n)$ has a PFD of degree $d$ if and only if $f$ vanishes to order at least $d-n+|S|$ on all flats $S$ of size at least $n-d+1$ of the hyperplane arrangement associated with the linear forms. 
\end{theorem}

\begin{theorem}\label{thm: PFD iff affine}
    Let $\ell_1,\ldots,\ell_n$ be $n$ nonzero linear polynomials in $R$, let $d\leq n$, and let $f\in R$ be a polynomial with $\deg(f)\geq d$. Then $f/(\ell_1\cdots\ell_n)$ has a PFD of degree $d$ if and only if $f$ vanishes to order at least $d-n+|S|$ on all flats $S$, where:
    \begin{itemize}
        \item $S$ is a flat of the hyperplane arrangement associated with the homogenizations of the $\ell_i$;
        \item the size of $S$ is at least $n-d+1$; and
        \item $S$ is not supported on the hyperplane at infinity.
    \end{itemize}
\end{theorem}

We now turn to the matter of computing PFDs. Recall the wishlist for a good PFD algorithm outlined in \S\ref{sec:intro}. We first focus on the notion of spurious poles.

\begin{definition}\label{defn:spur pole}
    A \emph{spurious pole} of a rational function is a denominator factor that can be removed upon a re-writing of the rational function.
\end{definition}

In Example~\ref{eg:intro}, the first decomposition of our rational function introduces a spurious pole at $x=0$; this is undesirable. Our notion of a PFD of degree $d$ guarantees that spurious poles are not introduced, as all of the poles in (\ref{eqn:pfd secalgo}) appear in both the original rational function and its PFD. However, before beginning any computation of PFDs of this form, we must ensure that spurious poles do not already exist in the input of $f/(\ell_1\cdots\ell_n)$.
To address this, consider the following definition: we say that a rational function $f/g$ with $f,g\in R$ nonzero polynomials is \emph{reducible} if $f,g$ share a common irreducible factor. Then, the input rational function $f/(\ell_1\cdots\ell_n)$ already having a spurious pole is equivalent to it being reducible. One also sees immediately that this is equivalent to $f\in\langle \ell_i\rangle$ for some $i\in[n]$. Algorithm~\ref{algo:reduced} below constructs a reduced expression for the rational function by finding and eliminating these spurious $\ell_i$ (if they exist).

\begin{algorithm}
	\caption{\texttt{ReducedExp}}
	\label{algo:reduced}
	\begin{flushleft}
		\textbf{Input:} rational function $f/(\ell_1\cdots\ell_n)$ where $f,\ell_1,\ldots,\ell_n\in K[x_1,\ldots,x_r]$ nonzero and $\deg(\ell_i)=1$
		
		\textbf{Output:} fully reduced expression for $f/(\ell_1\cdots\ell_n)$
	\end{flushleft}
	\begin{algorithmic}[1]
    \STATE { $L\leftarrow\{\}$}
    \FOR {$i$ from $1$ to $n$}
        \IF {$f\in\langle\ell_i\rangle$}
        \STATE {write $f=g\ell_i$ using \texttt{//}}
        \STATE {$f\leftarrow g$}
        \ELSE
        \STATE{add $\ell_i$ to $L$}
        \ENDIF
    \ENDFOR
    \RETURN {$f,L$, where the reduced rational function is given by $f/(\ell_{i_1}\cdots\ell_{i_m})$ with $\ell_{i_j}\in L$}
	\end{algorithmic}
\end{algorithm}

Once a reduced expression for $f/(\ell_1\cdots\ell_n)$ is obtained, one can proceed with the main PFD algorithm (Algorithm~\ref{algo:PFD}) to determine if the rational function has a PFD of degree $d$, with $d$ as large as possible. If such a PFD exists, then the algorithm returns one.

\begin{algorithm}
	\caption{\texttt{PFD}}
	\label{algo:PFD}
	\begin{flushleft}
		\textbf{Input:} fully reduced rational function $f/(\ell_1\cdots\ell_n)$ where $f,\ell_1,\ldots,\ell_n\in K[x_1,\ldots,x_r]$ nonzero and $\deg(\ell_i)=1$
		
		\textbf{Output:} PFD of degree $d\in[n]$ for $f/(\ell_1\cdots\ell_n)$ with $d$ maximum, if one exists
	\end{flushleft}
	\begin{algorithmic}[1]
        \STATE {$L \leftarrow (\ell_1,\ldots,\ell_n)$}
        \STATE {$d\leftarrow 0$}
        
        \FOR{$D$ from $1$ to $n$}
        \STATE {create ideal $I_{L,D}$}
        \IF {$f\in I_{L,D}$}
        \STATE {$d\leftarrow D$}
        \STATE {$C\leftarrow$ coefficients for quotient of $f$ with respect to generators of $I_{L,D}$ using \texttt{//} in \texttt{M2}}
        \ELSE
        \STATE {break}
        \ENDIF
        \ENDFOR
        \IF {$d=0$}
        \RETURN {no PFD of any degree possible}
        \ELSE
        \RETURN {PFD of degree $d$ in terms of $C$}
        \ENDIF
	\end{algorithmic}
\end{algorithm}
\newpage
\begin{remark}[Wishlist]
    We now explain why our algorithms satisfy the wishlist from \S\ref{sec:intro}.
\begin{itemize}
    \item[\ref{wishlist(i)}] \emph{Uniqueness of output:} Given a polynomial $f$ and an ordering $\ell_1,\ldots,\ell_n$ of the denominator factors of the rational function, our algorithm is deterministic in the output -- there are no probabilistic aspects of the algorithm. If we are in the case of a generic hyperplane arrangement and $f$ is homogeneous of degree $d$ (as in Theorem~\ref{thm: generic PFD}), then the output is also independent of any relabeling of the linear forms.
    \item[\ref{wishlist(ii)}] \emph{Does not introduce spurious poles:} This is guaranteed by the form of our output PFD (\ref{eqn:d PFD}), where denominator factors of the PFD are all subsets of the original $\ell_i$.
    \item[\ref{wishlist(iii)}] \emph{Commutes with summation:} This is guaranteed by the use of the \texttt{//} command in \texttt{Macaulay2}. This function uses polynomial division based on Gr\"obner bases, i.e., expressing polynomials in normal form with respect to the reduced Gr\"obner basis of the ideal $I_{L,d}$. The map sending polynomials to their normal form is a $K$-linear map.
    \item[\ref{wishlist(iv)}] \emph{Eliminates any spurious poles present in the input:} This is guaranteed by Algorithm~\ref{algo:reduced}.
\end{itemize}
\end{remark}

\begin{remark}
The algorithm above is highly customizable. Below we list examples of small changes that can be made to hone the algorithm for certain applications.
\begin{enumerate}
    \item \textit{Large $n$ or $r$:} Construction of the ideals $I_{L,d}$ in \texttt{Macaulay2} can take a long time when $n$ or $r$ is large. Below are two ways these large computation times can be circumvented.
    \begin{enumerate}
         \item  For a given choice of $d$, check if the numerator is in an ideal generated by some subset of the generators of $I_{L,d}$. If this is true, a PFD of degree $d$ is then obtainable with this restricted set of generators. This strategy is used below for the calculation of a ``large" coefficient coming from Feynman integrals (\S\ref{sssec:large feynman}).
        \item Suppose the numerator is in the ideal $I_{L,d}$ for some small choice of $d$. Rather than checking if the same numerator is in $I_{L,d}$ for larger $d$, consider the numerators in the PFD expansion for small $d$ and apply the PFD algorithm on them, also restricting to small choices of $d$ (if need be). Proceed iteratively. This iterative procedure will produce the same output as the original algorithm as a consequence of the fact that the algorithm commutes with summation \ref{wishlist(iii)}.
    \end{enumerate}
    \item \textit{Choosing PFD denominators:} Suppose one has a specific PFD form in mind, e.g. that does not have certain denominator factors. This can be achieved by restricting the generators of $I_{L,d}$ to the subset of them that do not produce the undesired denominator factors. If the numerator is a member of this restricted ideal, then the desired PFD can be generated. 
    \item \textit{Nonlinear denominator factors:} The algorithm above can readily be applied to the case where the $\ell_i$ are not linear. Unlike when all the $\ell_i$ are linear, the resulting PFD may contain terms with denominators of differing degrees.
\end{enumerate}    
\end{remark}

\section{Towards physics}\label{sec:physics}

In this section, we apply our algorithm from \S\ref{sec: PFDs} to compute PFDs of rational functions arising from two different contexts: Feynman integrals (\S\ref{subsec:feynman}) and flat-space wavefunction coefficients (\S\ref{subsec:wavefunc}). The code for these examples can be found \href{https://github.com/CJMDK/PFDs}{here}.

\subsection{Feynman integrals}\label{subsec:feynman}
In preparation for the approaching upgrade to the Large Hadron Collider \cite{AAA+}, much work has been done in developing efficient algorithms to aid the computation of Feynman integrals \cite{HvM,BFP,CGM}. Many of these algorithms are based on \emph{integration-by-parts} (IBP) \textit{reductions}, in which a Feymann integral is decomposed as a linear combination of \emph{master integrals} (MIs). These MIs are a choice of basis for the family of Feynman integrals under consideration. The coefficients of the linear combination are rational functions in the external kinematics and the dimensional regularization parameter $\epsilon$. These rational functions are often very large, and can make calculations cumbersome. The authors of \cite{BWW+} suggest storing the coefficients of IBP reductions in the form of a PFD, and introduce their own PFD algorithm as a ``powerful method to reduce the byte size of the analytic IBP reduction coeﬃcients". Their paper includes an application to IBP reduction coefficients arising from the Feynman integral associated with the \emph{two-loop five-point non-planar double pentagon diagram} shown in Figure \ref{fig: Feynman diagram}. These coefficients were calculated (with respect to a specific choice of MIs) in \cite{BBD+}.
The authors of \cite{HvM} suggest using these coefficients as a benchmark for testing PFD algorithms for Feynman integrals. In this section we naively apply our PFD algorithm to two of these coefficients, a ``small" and a ``large" example, and discuss the results. We hope this may serve as a starting point for further refinement.

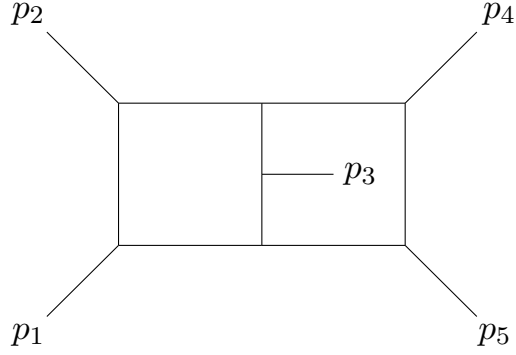
\begin{figure}[h]
\centering
\begin{tikzpicture}[scale=0.2ex]
\draw[black] (0,0) -- (-1,-1) node[scale=1.2, xshift=-1.2ex, yshift=-1.2ex]{$p_1$};
\draw[black] (0,0) -- (2,0);
\draw[black] (0,0) -- (0,2);
\draw[black] (0,2) -- (-1,3) node[scale=1.2, xshift=-1.2ex, yshift=1.2ex]{$p_2$};
\draw[black] (0,2) -- (2,2);
\draw[black] (2,2) -- (4,2);
\draw[black] (4,2) -- (5,3) node[scale=1.2, xshift=1.5ex, yshift=1.2ex]{$p_4$};
\draw[black] (4,2) -- (4,0);
\draw[black] (4,0) -- (5,-1) node[scale=1.2, xshift=1.2ex, yshift=-1.2ex]{$p_5$};
\draw[black] (4,0) -- (2,0);
\draw[black] (2,2) -- (2,1);
\draw[black] (2,1) -- (2,0);
\draw[black] (2,1) -- (3,1) node[scale=1.2, xshift=1.8ex, yshift=0]{$p_3$};
\end{tikzpicture}
\caption{Two-loop five-point non-planar double pentagon Feynman diagram with labelled external momenta $p_i$.}
\label{fig: Feynman diagram}
\end{figure}

\subsubsection{Small IBP coefficient}
This IBP coefficient is a rational function in six variables. These variables are the dimensional regularization parameter $\epsilon$ and the Mandelstam variables $s_{12}$, $s_{15}$, $s_{23}$, $s_{34}$, and $s_{45}$ (Lorentz invariant functions of the external momenta $p_i$ shown in Figure \ref{fig: Feynman diagram}). 
Explicitly, the polynomial ring here is $\mathbb{Q}[s_{12},s_{15},s_{23},s_{34},s_{45},\epsilon]$.
The numerator is a degree $11$ polynomial with $785$ terms. This polynomial can be accessed through the links provided in \cite{BBD+}. The denominator can be factorized as a product of the hyperplanes:
\begin{equation*}
    \begin{aligned}
    \ell_1&=8(1+4\epsilon)\\
    \ell_4&=s_{12}+s_{15}-s_{34}\\
    \ell_7&=s_{12}+s_{23}-s_{45}\\
    \ell_{10}&=s_{45}\\
\end{aligned}
\quad\quad
\begin{aligned}
    \ell_2&=s_{12}\\
    \ell_5&=-s_{15}+s_{23}+s_{34}\\
    \ell_8&=-s_{15}+s_{23}-s_{45}\\
    \ell_{11}&=s_{45}\\
\end{aligned}
\quad\quad
\begin{aligned}
    \ell_3&=s_{23}\\
    \ell_6&=s_{12}-s_{45}\\
    \ell_9&=s_{12}-s_{34}-s_{45}\\
    \ell_{12}&=s_{34}+s_{45}.
\end{aligned}
\end{equation*}
We find that the input rational function contains no spurious denominator factors and that the largest $d$ for which the numerator is in the ideal $I_{L,d}$ is $8$. This implies the existence of a PFD of degree $8$ as in (\ref{eqn:pfd secalgo}), with degree $4$ denominators. Algorithm \ref{algo:PFD} produces such a PFD with $173$ terms (where the number of distinct terms is the number of distinct denominators appearing in the PFD). These results were computed within seconds. 

\subsubsection{Large IBP coefficient}\label{sssec:large feynman}
This IBP coefficient is a rational function in the same six variables as above. To reduce computation time, one of the Mandelstam variables can be set to $1$. This variable can later be re-introduced through dimensional arguments. We choose to set $s_{12}=1$, and so the polynomial ring here is $\mathbb{Q}[s_{15},s_{23},s_{34},s_{45},\epsilon]$. The numerator is then a degree $30$ polynomial in five variables with $73763$ terms. As before, it can be accessed through the links in \cite{BBD+}. The denominator can be written as a product of the following $29$ hyperplanes:
\begin{equation*}
    \begin{aligned}
    \ell_1&=8(-1+\epsilon)\\
    \ell_4&=-1+4\epsilon\\
    \ell_7&=s_{15}-s_{23}\\
    \ell_{10}&=s_{12}+s_{23}\\
    \ell_{13}&=s_{15}-s_{23}-s_{34}\\
    \ell_{16}&=s_{12}-s_{45}\\
    \ell_{19}&=s_{23}-s_{45}\\
    \ell_{22}&=s_{12}-s_{34}-s_{45}\\
    \ell_{25}&=s_{45}\\
    \ell_{28}&=s_{34}+s_{45}\\
\end{aligned}
\quad\quad
\begin{aligned}
    \ell_2&=-1+2\epsilon\\
    \ell_5&=s_{15}\\
    \ell_8&=s_{23}\\
    \ell_{11}&=s_{12}+s_{15}-s_{34}\\
    \ell_{14}&=s_{23}+s_{34}\\
    \ell_{17}&=s_{12}-s_{45}\\
    \ell_{20}&=s_{12}+s_{23}-s_{45}\\
    \ell_{23}&=s_{12}+s_{15}-s_{34}-s_{45}\\
    \ell_{26}&=s_{15}-s_{23}+s_{45}\\
    \ell_{29}&=s_{34}+s_{45}.\\
\end{aligned}
\quad\quad
\begin{aligned}
    \ell_3&=-1+3\epsilon\\
    \ell_6&=s_{15}-s_{23}\\
    \ell_9&=s_{12}+s_{23}\\
    \ell_{12}&=s_{12}+s_{15}-s_{34}\\
    \ell_{15}&=s_{23}+s_{34}\\
    \ell_{18}&=s_{23}-s_{45}\\
    \ell_{21}&=s_{12}-s_{34}-s_{45}\\
    \ell_{24}&=s_{12}+s_{15}-s_{34}-s_{45}\\
    \ell_{27}&=s_{15}-s_{23}+s_{45}\\
\end{aligned}
\end{equation*}
Given that we have set $s_{12}=1$, we have suppressed the $30$th denominator hyperplane $\ell_{30}=s_{12}$. Without making any modifications to Algorithm \ref{algo:PFD}, the time it takes to construct the ideals $I_{L,d}$ grows exponentially with $d$, already taking hours at $d=7$. By choosing subsets of the generators of the ideals one can dramatically reduce computation time. The only drawback is that the numerator not being in a subset of $I_{L,d}$ does not rule out the non-existence of a PFD of degree $d$. Using this strategy, we have found that the numerator is in an ideal contained in the ideal $I_{L,17}$. This smaller ideal was generated using $5985$ of the $51895935$ generators of $I_{L,17}$ (we did not include $\ell_{22},\ldots,\ell_{29}$) and took $194$ seconds to construct. Testing that the numerator is in this ideal took $930$ seconds, and generating the corresponding PFD took $2217$ seconds. This PFD has $546$ terms (counted by distinct denominators) and degree $12$ denominators. Comparing with the times given in \cite[Table 2]{HvM}, our algorithm implemented in \texttt{Macaulay2} is competitive with other state-of-the-art PFD algorithms for Feynman integrals.

\subsection{Wavefunction coefficients}\label{subsec:wavefunc}

As mentioned in \S\ref{sec:intro}, one of the ambitions of modern theoretical physics is to develop an alternative picture of scattering amplitudes, one that moves away from Feynman integrals and takes a more geometric approach. Landmark papers in this direction include \cite{AT}, and show that certain scattering amplitudes can be represented in terms of canonical forms of polytopes such as the amplituhedron and associahedron. Similar developments have been made in the study of cosmological correlators, relating flat-space wavefunction coefficients and canonical forms of cosmological polytopes \cite{ABP}. Below we show an example of how Algorithm~\ref{algo:PFD} can be used to shed light on the structure of flat-space wavefunction coefficients.

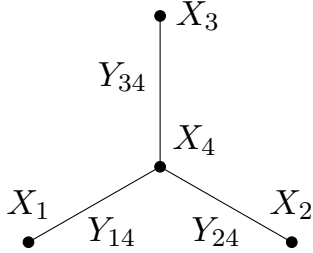
\begin{figure}[h]
\centering
\begin{tikzpicture}
\draw[black] (0,2) -- (0,0) node[scale=1.2, yshift=6ex, xshift=-2.5ex]{$Y_{34}$};
\draw[black] (1.732,-1) -- (0,0) node[scale=1.2, yshift=-4.3ex, xshift=3.7ex]{$Y_{24}$};
\draw[black] (-1.732,-1) -- (0,0) node[scale=1.2, yshift=-4.3ex, xshift=-3.2ex]{$Y_{14}$};
\filldraw[black] (0,0) circle (2pt) node[scale=1.2, above right]{$X_4$};
\filldraw[black] (0,2) circle (2pt) node[scale=1.2, xshift=2.5ex]{$X_3$};
\filldraw[black] (1.732,-1) circle (2pt) node[scale=1.2, yshift=2.5ex]{$X_2$};
\filldraw[black] (-1.732,-1) circle (2pt) node[scale=1.2, yshift=2.5ex]{$X_1$};
\end{tikzpicture}
\caption{Graph corresponding to the flat-space wavefunction coefficient $\psi$.}
\label{fig: wavefunction}
\end{figure}

\begin{example}
As explained in \cite{ABHJLP}, the flat-space wavefunction coefficient $\psi$ associated with the graph shown in Figure \ref{fig: wavefunction} can be generated by considering tubings of the graph. The result is the following expression for $\psi$:
\begin{equation*}
\psi=\frac{8 Y_{14}Y_{24}Y_{34}}{\ell_1\ell_2\ell_3\ell_4\ell_5}\bigg[\frac{1}{\ell_9}\bigg(\frac{1}{\ell_6}+\frac{1}{\ell_7}\bigg)+\frac{1}{\ell_{10}}\bigg(\frac{1}{\ell_6}+\frac{1}{\ell_8}\bigg)+\frac{1}{\ell_{11}}\bigg(\frac{1}{\ell_7}+\frac{1}{\ell_8}\bigg)\bigg],
\end{equation*}
\begin{equation*}
\begin{aligned}
    \ell_1&=X_1+Y_{14}\\
    \ell_3&=X_3+Y_{34}\\
    \ell_5&=X_1+X_2+X_3+X_4\\
    \ell_7&=X_2+X_4+Y_{14}+Y_{34}\\
    \ell_9&=X_1+X_2+X_4+Y_{34}\\
    \ell_{11}&=X_2+X_3+X_4+Y_{14}.\\
\end{aligned}
\quad\quad
\begin{aligned}
    \ell_2&=X_2+Y_{24}\\
    \ell_4&=X_4+Y_{14}+Y_{24}+Y_{34}\\
    \ell_6&=X_1+X_4+Y_{24}+Y_{34}\\
    \ell_8&=X_3+X_4+Y_{14}+Y_{24}\\
    \ell_{10}&=X_1+X_3+X_4+Y_{24}
\end{aligned}
\end{equation*}
The function $\psi/(8Y_{14}Y_{24}Y_{34})$ is the canonical form of the cosmological polytope whose facet hyperplanes are the denominator factors. In \cite{FPSW}, the authors conjecture a formula for the PFD of flat-space wavefunction coefficients associated with arbitrary tree-level cosmological Feynman diagrams. An amended version of this conjecture was recently proven in \cite{Dun}. That a PFD of this form exists in this case is clear from the fact that the numerator of $\psi$, a polynomial of degree $7$ with $204$ terms (once all the terms have been collected under a common denominator), is in the ideal:
\begin{equation*}
    I_{L,7}=\langle\ell_{i_1}\cdots\ell_{i_7}:1\le i_1<i_2<\cdots<i_7\le11\rangle,
\end{equation*}
(as can be checked using \texttt{Macaulay2}), and that the numerator is homogeneous of degree $7$. And so, using the command \texttt{//} one can express $\psi$ in terms of the generators of this ideal, yielding:
\begin{equation*}
\begin{split}
\psi &=\frac{1}{\ell_7\ell_9\ell_{10}\ell_{11}}+\frac{1}{\ell_5\ell_9\ell_{10}\ell_{11}}+\frac{1}{\ell_4\ell_8\ell_{10}\ell_{11}}+\frac{1}{\ell_4\ell_7\ell_{10}\ell_{11}}+\frac{1}{\ell_{4}\ell_{7}\ell_{9}\ell_{10}}+\frac{1}{\ell_{4}\ell_{6}\ell_{9}\ell_{10}}+\frac{1}{\ell_{3}\ell_{8}\ell_{10}\ell_{11}}\\& \ \ \ +\frac{1}{\ell_{3}\ell_{5}\ell_{10}\ell_{11}}-\frac{1}{\ell_{3}\ell_{4}\ell_{7}\ell_{9}}-\frac{1}{\ell_{3}\ell_{4}\ell_{6}\ell_{9}}+\frac{1}{\ell_{2}\ell_{7}\ell_{9}\ell_{11}}+\frac{1}{\ell_{2}\ell_{5}\ell_{9}\ell_{11}}-\frac{1}{\ell_{2}\ell_{4}\ell_{8}\ell_{10}}-\frac{1}{\ell_{2}\ell_{4}\ell_{6}\ell_{10}}\\& \ \ \ -\frac{1}{\ell_{2}\ell_{3}\ell_{8}\ell_{10}}-\frac{1}{\ell_{2}\ell_{3}\ell_{7}\ell_{9}}+\frac{1}{\ell_{2}\ell_{3}\ell_{5}\ell_{11}}+\frac{1}{\ell_{2}\ell_{3}\ell_{4}\ell_{6}}+\frac{1}{\ell_{1}\ell_{6}\ell_{9}\ell_{10}}+\frac{1}{\ell_{1}\ell_{5}\ell_{9}\ell_{10}}-\frac{1}{\ell_{1}\ell_{4}\ell_{8}\ell_{11}}\\& \ \ \ -\frac{1}{\ell_{1}\ell_{4}\ell_{7}\ell_{11}}-\frac{1}{\ell_{1}\ell_{3}\ell_{8}\ell_{11}}-\frac{1}{\ell_{1}\ell_{3}\ell_{6}\ell_{9}}+\frac{1}{\ell_{1}\ell_{3}\ell_{5}\ell_{10}}+\frac{1}{\ell_{1}\ell_{3}\ell_{4}\ell_{7}}-\frac{1}{\ell_{1}\ell_{2}\ell_{7}\ell_{11}}-\frac{1}{\ell_{1}\ell_{2}\ell_{6}\ell_{10}}\\& \ \ \ +\frac{1}{\ell_{1}\ell_{2}\ell_{5}\ell_{9}}+\frac{1}{\ell_{1}\ell_{2}\ell_{4}\ell_{8}}-\frac{1}{\ell_{1}\ell_{2}\ell_{3}\ell_{11}}-\frac{1}{\ell_{1}\ell_{2}\ell_{3}\ell_{10}}-\frac{1}{\ell_{1}\ell_{2}\ell_{3}\ell_{9}}+\frac{1}{\ell_{1}\ell_{2}\ell_{3}\ell_{8}}+\frac{1}{\ell_{1}\ell_{2}\ell_{3}\ell_{7}}\\& \ \ \ +\frac{1}{\ell_{1}\ell_{2}\ell_{3}\ell_{6}}+\frac{1}{\ell_{1}\ell_{2}\ell_{3}\ell_{5}}-\frac{1}{\ell_{1}\ell_{2}\ell_{3}\ell_{4}}.
\end{split}
\end{equation*}
Note that this is \emph{not} the PFD one arrives at for Figure~\ref{fig: Feynman diagram} using the process outlined in \cite{Dun}, though both PFDs have the same number of terms. This means that there can be multiple different PFDs of tree-level flat space wavefunction coefficients which have the same number of terms, and in which all numerators are $\pm1$ and all denominators have the same degree. In other words, such ``minimal" flat-space wavefunction coefficient PFDs are not unique; this may be related to the different triangulations of cosmological polytopes found in \cite{BDFJ}.

One may also wonder about the vanishing properties of the numerator. Using the \texttt{Matroids} package for \texttt{Macaulay2}, one finds that the matroid defined by lines $\ell_i$ has $88$ flats of five elements, $44$ of six elements, $12$ of seven elements, $1$ of eight elements, $3$ of nine elements, and $3$ of ten elements. Theorem \ref{thm:flats-pri-decomp} then tells us that the numerator vanishes on each of these flats $S$, with order $|S|-4$. For example, the flats of ten elements are:
\begin{equation*}
\begin{split}
    &\{\ell_{2},\ell_{3},\ell_{4},\ell_{5},\ell_{6},\ell_{7},\ell_{8},\ell_{9},\ell_{10},\ell_{11}\},\\
    &\{\ell_{1},\ell_{3},\ell_{4},\ell_{5},\ell_{6},\ell_{7},\ell_{8},\ell_{9},\ell_{10},\ell_{11}\},\\
    &\{\ell_{1},\ell_{2},\ell_{4},\ell_{5},\ell_{6},\ell_{7},\ell_{8},\ell_{9},\ell_{10},\ell_{11}\}.
\end{split}
\end{equation*}
\end{example}

\section{Future directions and outlook}\label{sec:future}

We present some potential directions for future work regarding PFDs.

In the generic case, PFDs of degree $d$ are unique (Theorem~\ref{thm: generic PFD}). But given arbitrary $f\in R$ and an arbitrary hyperplane arrangement $L$, there can be be many different such PFDs. This is because the products of $d$ of the $\ell_i$ need not be algebraically independent. To see this explicitly, consider the following example.
\begin{example}
Define the hyperplanes $\ell_1 = x-y, \ \ell_2=y+z, \ \ell_3=z, \ \ell_4=x+y\in K[x,y,z]$, and the polynomial $f$:
    \begin{align*}
        f &= 2x^2y^2-2y^4+12x^2yz+4xy^2z-2y^3z+10x^2z^2+4xyz^2\nonumber\\
        &=3x\ell_1\ell_2\ell_3+2y\ell_1\ell_2\ell_4+7x\ell_2\ell_3\ell_4\nonumber\\
        &=(-2x-5y)\ell_1\ell_2\ell_3+2y\ell_1\ell_2\ell_4+10\ell_1\ell_2\ell_3\ell_4+(2x+5y)\ell_2\ell_3\ell_4.
    \end{align*}
    We then have the following two PFDs of degree $3$:
    \begin{equation}\label{eqn:2pfd}
        \frac{f}{\ell_1\ell_2\ell_3\ell_4}=\frac{3x}{x+y}+\frac{2y}{z}+\frac{7x}{x-y}=-\frac{2x+5y}{x+y}+\frac{2y}{z}+10+\frac{2x+5y}{x-y}.
    \end{equation}
\end{example}

Our algorithm gives a unique partial fraction decomposition, but changing the choice of labeling or ordering of the $\ell_i$ could result in different PFDs.

\begin{question}
    Given a hyperplane arrangement $L$ and a polynomial $f\in I_{L,d}$, can one determine the number of unique PFDs of degree $d$?
\end{question}

Another question one can ask concerns when a term in a PFD is reducible, such as in the second PFD of (\ref{eqn:2pfd}), where the term $10$ appears.

\begin{question}
    What are necessary or sufficient conditions for when a term of a PFD of degree $d$ is reducible, or even a polynomial?
\end{question}

A first step in answering this question is the following result.

\begin{lemma}\label{lemma: reducibility}
Suppose $d=n-1\geq 1$, all of the $\ell_i$ are pairwise prime linear polynomials, and $f\in I_{L,d}$ is a polynomial. Then a PFD of degree $d$ for $f/(\ell_1\cdots\ell_n)$ has a reducible term if and only if $f\in\langle \ell_i\rangle$ for some $i\in[n]$.
\end{lemma}
\begin{proof}
Since $f\in I_{L,d}$, it can be written in the form $f=\sum_{j=1}^na_j\ell_{1}\cdots\hat{\ell}_j\cdots\ell_{n}$ for some $a_{j}\in R$. Then we have the following chain of equivalences:
\begin{align}
& \ \ \ \ \ f\in\langle\ell_i\rangle\text{ for some } i\in[n]\nonumber\\
&\Leftrightarrow f=\ell_ig\text{ for some } g\in R\nonumber\\
&\Leftrightarrow \sum_{j=1}^na_j\ell_{1}\cdots\hat{\ell}_j\cdots\ell_{n}=\ell_ig\nonumber\\
& \Leftrightarrow a_i=\ell_ih\text{ for some } h\in R,\text{ since the $\ell_i$ are pairwise prime}
\nonumber\\
&\Leftrightarrow \frac{f}{\ell_1\cdots\ell_n}=h+\sum_{j=1,j\neq i}^n \frac{a_{j}}{\ell_j}\nonumber.
\end{align}
\end{proof}

By recasting the existence of PFDs as an ideal membership problem, we have proven a number of results about PFDs of rational functions on hyperplane arrangements. This led to new ways of understanding the geometry of adjoint polynomials and canonical forms, as well as a novel algorithmic approach for constructing PFDs of a tailored form. We have demonstrated how our findings have direct applications to a variety of problems in high-energy physics and hope that this paper inspires further study of PFDs by algebraists.

\subsection*{Acknowledgements}
We thank Bernd Sturmfels, Andreas von Manteuffel, and Julian Weigert  for valuable feedback and guidance during the development of this project. CdK was supported by the European Research Council through the synergy grant UNIVERSE+, 101118787. 

\noindent {\tiny Views and opinions
expressed are however those of the authors only and do not necessarily reflect those of the European Union or the European
Research Council Executive Agency. Neither the European Union nor the granting authority can be held responsible for them.}

\vspace{1em}

\noindent {\bf Authors' addresses:}

\noindent 
Claire de Korte, MPI-MiS Leipzig, Germany\hfill{\texttt{claire.dekorte@mis.mpg.de}}\\
Teresa Yu, MPI-MiS Leipzig, Germany \hfill{\texttt{teresa.yu@mis.mpg.de}}


\begin{thebibliography}{LNNR2}


\bibitem[AAA+]{AAA+}
A. Abada, M. Abbrescia, S.S. AbdusSalam, et al. HE-LHC: The High-Energy Large Hadron Collider. \textit{Eur. Phys. J. Spec. Top.} \textbf{228} (2019), pp.~1109–-1382. 

\bibitem[ABH+]{ABHJLP}
Nima Arkani-Hamed, Daniel Baumann, Aaron Hillman,
Austin Joyce, Hayden Lee, Guilherme L. Pimentel. Differential Equations for Cosmological Correlators. \textit{J. High Energ. Phys.} \textbf{2025} (2025), no.~9. \arxiv{2312.05303v2}

\bibitem[ABL]{ABL}
Nima Arkani-Hamed, Yuntao Bai, Thomas Lam. Positive Geometries and Canonical Forms. \textit{J. High Energ. Phys.} \textbf{2017} (2017), no.~39. \arxiv{1703.04541v2}

\bibitem[ABP]{ABP}
Nima Arkani-Hamed, Paolo Benincasa, Alexander Postnikov. Cosmological Polytopes and the Wavefunction of the Universe. \arxiv{1709.02813v1}

\bibitem[AT]{AT}
Nima Arkani-Hamed, Jaroslav Trnka. The Amplituhedron. \textit{J. High Energ. Phys.} \textbf{2014} (2014), no.~30. \arxiv{1312.2007v1}


\bibitem[BBD+]{BBD+}
Dominik Bendle, Janko B\"ohm, Wolfram Decker, Alessandro Georgoudis, Franz-Josef Pfreundt, Mirko Rahn, Pascal Wasser, Yang Zhang. Integration-by-parts reductions of Feynman integrals using Singular and GPI-Space. \textit{J. High Energ. Phys.} \textbf{2020} (2020), no.~79. \arxiv{1908.04301}

\bibitem[BDFJ]{BDFJ}
Anna Birkemeyer, Torben Donzelmann, Mieke Fink, Martina Juhnke. On cosmological polytopes, their canonical forms and their duals. \arxiv{2603.03894}

\bibitem[Bec]{Bec}
Matthias Beck. The partial-fractions method for counting solutions to integral linear systems. \textit{Discrete Comput. Geom.} \textbf{32} (2004), no.~4, pp.~437--446. \arxiv{math/0309332}

\bibitem[BEPV]{BEPV}
Sarah Brauner, Christopher Eur, Elizabeth Pratt, Raluca Vlad. Wondertopes. \textit{Adv. Math.} \textbf{480} (2025), part C, Paper No.~110516. \arxiv{2403.04610v2}

\bibitem[BFP]{BFP}
Giuseppe Bertolini, Gaia Fontana, Tiziano Peraro. CALICO: Computing Annihilators from Linear Identities Constraining (differential) Operators. \textit{J. High Energ. Phys.} \textbf{2025} (2025), no.~18. \arxiv{2506.13653v1}

\bibitem[Bry]{Bry}
Thomas Brylawski. The lattice of integer partitions. \textit{Discrete Math.} \textbf{6} (1973), pp.~201--219.

\bibitem[BTX]{BTX}
Ricardo Burity, \cb{S}tefan O. Toh\u{a}neanu, Yu Xie. Homological properties of ideals generated by fold products of linear forms. \textit{Michigan Math. J.} \textbf{74} (2024), no.~4, pp.~797--824. \arxiv{2004.07430}

\bibitem[BW]{BW}
Clemens Br\"user, Julian Weigert. Geometry of Adjoint Hypersurfaces for Polytopes. \arxiv{2511.13537} 

\bibitem[BW+]{BWW+}
Janko Boehm, Marcel Wittmann, Zihao Wu, Yingxuan Xu, Yang Zhang. IBP reduction coefficients made simple. \textit{J. High Energ. Phys.} \textbf{2020} (2020), no.~54. \arxiv{2008.13194}

\bibitem[CDE]{CDE}
John M. Campbell, Giuseppe De Laurentis, R. Keith Ellis. Analytic reconstruction with massive particles: one-loop amplitudes for $0\rightarrow \overline{q}qt\overline{t}H$. \textit{J. High Energ. Phys.} \textbf{2025} (2025), no.~147. \arxiv{2504.19909}

\bibitem[CGM]{CGM}
Miguel Correia, Mathieu Giroux, Sebastian Mizera. SOFIA: Singularities of Feynman Integrals Automatized. \textit{Comput. Phys. Commun.} \textbf{320}, (2026). \arxiv{2503.16601v2}

\bibitem[CH]{CH}
Aldo Conca, J\"urgen Herzog. Castelnuovo--Mumford regularity of products of ideals. \textit{Collect.\ Math.} \textbf{54} (2003), no. 2, pp.~137--152. \arxiv{math/0210065}

\bibitem[CT]{CT}
Aldo Conca, Manolis C. Tsakiris. Resolution of ideals associated to subspace arrangements. \textit{Algebra\ Number\ Theory} \textbf{16} (2022), no. 5, pp.~1121--1140. \arxiv{1910.01955}

\bibitem[DP]{DP}
Giuseppe De Laurentis, Ben Page. Ans\"atze for Scattering Amplitudes from $p$-adic Numbers and Algebraic Geometry. \textit{J. High Energ. Phys.}, \textbf{2022} (2022), no.~140. \arxiv{2203.04269}

\bibitem[Dun]{Dun}
Tyler Dunaisky. Representations of the Flat Space Wavefunction. \arxiv{2601.17619v1}


\bibitem[Fei]{Fei}
Eva Maria Feichtner. De Concini--Procesi wonderful arrangement models: a discrete geometer's point of view. \textit{Combinatorial and computational geometry,} pp.~333--360. Math.~Sci.~Res.~Inst.~Publ., 52. \textit{Cambridge University Press, Cambridge,} 2005.

\bibitem[FPSW]{FPSW}
Claudia Fevola, Guilherme L. Pimentel, Anna-Laura Sattelberger, Tom Westerdijk. Algebraic Approaches to Cosmological Integrals. \textit{Le Matematiche (Catania)} \textbf{80} (2025), no.~1, pp.~303--324. \arxiv{2410.14757v2}

\bibitem[Fro]{Fro}
Hadleigh Frost. The Algebraic Structure of the KLT Relations for Gauge and Gravity Tree Amplitudes. \textit{SIGMA Symmetry Integrability Geom. Methods Appl.} \textbf{17} (2021), Paper No.~101. \arxiv{2111.07257}

\bibitem[Gae]{Gae}
Christian Gaetz. Canonical forms of polytopes from adjoints. \arxiv{2504.07272v1}

\bibitem[Hod]{Hod}
Andrew Hodges. Eliminating spurious poles from gauge-theoretic amplitudes. \textit{J. High Energ. Phys.} \textbf{2013} (2013), no.~135. \arxiv{0905.1473v1}

\bibitem[HvM]{HvM}
Matthias Heller, Andreas von Manteuffel. MultivariateApart: Generalized Partial Fractions. \textit{Comput. Phys. Commun.} \textbf{271} (2022), Paper No. 108174. \arxiv{2101.08283}

\bibitem[KR]{KR}
Kathl\'en Kohn, Kristian Ranestad. Projective geometry of Wachspress coordinates. \textit{Found. Comp. Math.} \textbf{20} (2020), no.~5, pp.~1135--1173. \arxiv{1904.02123}

\bibitem[Lam]{Lam}
Thomas Lam. An invitation to positive geometries. \textit{Open Problems in Algebraic Combinatorics}, pp.~159--179. Proc. Sympos. Pure Math., 110, (2024). \arxiv{2208.05407}

\bibitem[Lei]{Lei}
E. K. Le\u{\i}nartas. Factorization of rational functions of several variables into partial fractions. \textit{Izv. Vyssh. Uchebn. Zaved. Mat.} \textbf{197} (1978), no.~10, pp.~47--51.

\bibitem[M2]{M2}
Daniel R. Grayson, Michael E. Stillman. \texttt{Macaulay2}, a software system for research in algebraic geometry. Available at \url{http://www2.macaulay2.com}.

\bibitem[RST]{RST}
Kristian Ranestad, Bernd Sturmfels, Simon Telen. What is Positive Geometry? \textit{Le Matematiche (Catania)} \textbf{80} (2025), no.~1, pp.~3--16. \arxiv{2502.12815}

\bibitem[Toh]{Toh}
\cb{S}tefan O. Toh\u{a}neanu. On the de Boer--Pellikaan method for computing minimum distance. \textit{J. Symbolic Comput.} \textbf{45} (2010), no.~10, pp.~965--974.


\bibitem[War]{War}
Joe Warren, Barycentric coordinates for convex polytopes. \textit{Adv. Comput. Math.} \textbf{6} (1996), pp.~97--108.












\end{thebibliography}
\end{document}